\documentclass[11pt]{article}
\usepackage[a4paper,total={155mm,237mm},left=27.5mm,top=30mm]{geometry}
\usepackage[T1]{fontenc}
\usepackage{lmodern,microtype}
\usepackage{amssymb,amsthm,amsmath,mathtools}
\usepackage{graphicx,xcolor,leftindex}
\usepackage{authblk}
\usepackage[colorlinks=true,linkcolor=blue!50!black,citecolor=blue!50!black,urlcolor=blue!50!black]{hyperref}

\DeclareMathOperator{\pos}{pos}

\newcommand{\R}{\mathbb{R}}
\newcommand{\ip}[2]{\langle #1,#2\rangle}
\newcommand{\norm}[1]{\lVert #1\rVert}

\newcommand{\eps}{\varepsilon}

\newcommand{\PP}{\mathcal P}
\newcommand{\SOn}[1]{\operatorname{SO}(#1)}
\newcommand{\On}[1]{\operatorname{O}(#1)}
\newcommand{\CC}{\mathcal{C}}
\newcommand{\sphere}[1]{\mathbb{S}^{#1}}
\newcommand{\lin}{\operatorname{lin}}
\DeclareMathOperator{\sgn}{sgn}
\DeclareMathOperator{\diag}{diag}
\DeclareMathOperator{\inter}{int}
\DeclareMathOperator{\relint}{relint}
\newtheorem{theorem}{Theorem}[section]
\newtheorem{proposition}[theorem]{Proposition}
\newtheorem{corollary}[theorem]{Corollary}
\newtheorem{lemma}[theorem]{Lemma}

\hypersetup{pdftitle={Hadwiger's classification theorem on the sphere via signed orthoscheme decompositions},pdfauthor={Dennis Amelunxen and Martin Lotz}}

\title{Hadwiger's classification theorem on the sphere via signed orthoscheme decompositions}
\author{Martin Lotz}
\affil{Mathematics Institute, University of Warwick}

\begin{document}

\maketitle

\begin{abstract}
We give an elementary and combinatorial proof of the spherical Hadwiger theorem, which states that every continuous, rotation-invariant valuation on the space of closed convex cones in $\R^n$ is a linear combination of the conic intrinsic volumes.
Our proof uses a signed orthoscheme decomposition. Its vanishing step requires only integrability of an orthoscheme restriction of the residual valuation, a weaker regularity condition than continuity.
\end{abstract}


\section{Introduction}
Hadwiger's theorem classifies the continuous, rigid-motion-invariant
valuations on convex bodies in Euclidean space as linear combinations of
the intrinsic volumes~\cite{Had57:Vorlesungen,klain:95}. Its spherical
counterpart, asking for the classification of continuous, rotation-invariant
valuations on spherically convex sets, or equivalently on closed convex
cones, was conjectured by McMullen~\cite[Problem~49]{Gruber1979} and was 
only recently settled by Knoerr~\cite[Theorem~A]{Knoerr:Spherical-Hadwiger} and
independently by Wang and Wu~\cite{WangWu:Spherical-Hadwiger}.
The theorem states that every continuous, rotation-invariant valuation on the space of closed convex cones in $\R^n$ is a linear combination of the conic intrinsic volumes.
The conic intrinsic volumes arise naturally in spherical integral
geometry and in applications to optimization, statistics and compressive
sensing, see~\cite{Sch22:Convex-Cones} and the references therein.

In this paper, we give an elementary proof of the spherical Hadwiger theorem
through signed orthoscheme decompositions. In line with our previous expository article~\cite{AmelunxenLotz:Intrinsic-Volumes-Cones},
the proof is based on elementary conic geometry. The vanishing step uses a weaker regularity condition
than continuity, requiring only integrability of an orthoscheme restriction of the residual valuation.

\begin{theorem}[Spherical Hadwiger theorem]\label{thm:main}
Let $n\ge2$. Every continuous $\SOn n$-invariant real valuation $\nu$ on
$\CC(\R^n)$ has the unique representation
\begin{equation}\label{eq:main-classification}
 \nu(C)=\sum_{i=0}^n\nu(L_i)v_i(C),\qquad C\in\CC(\R^n),
\end{equation}
where $L_i$ is any $i$-dimensional linear subspace and $v_i$ is the $i$th
conic intrinsic volume.
\end{theorem}

Using a standard reduction, Theorem~\ref{thm:main} is equivalent 
to the statement that every continuous, $\On n$-invariant 
valuation that vanishes on lower-dimensional cones or cones containing a line, vanishes
on all simplicial cones (Theorem~\ref{thm:simplicial-vanishing}). Let $\mu$ denote such a valuation. The remaining proof establishes this
vanishing through orthoschemes.
An orthoscheme cone is generated by the successive orthogonal
projections of an apex onto a complete flag of subspaces. In standard
coordinates, its generators are
\[
 (a_1,0,\ldots,0),\ (a_1,a_2,0,\ldots,0),\ \ldots,\ (a_1,\ldots,a_n).
\]
Write $\psi(a)$ for its value under $\mu$, $\psi(a)=\mu(C(a))$, where $C(a)$ is the
cone generated by the vectors above, with $a\in\R_{>0}^n$. Fix a full-dimensional simplicial cone $C$.
For each permutation $\pi\in S_n$, let $Q_\pi$ be the orthonormal matrix
obtained by Gram--Schmidt from the corresponding ordering of the extreme
rays of $C$. For generic $u\in\sphere{n-1}$, we derive the identity
\begin{equation}\label{eq:intro-flag}
 \mu(C)=\sum_{\pi\in S_n}
 \left(\prod_{k=1}^n\sgn\bigl((Q_\pi^Tu)_k\bigr)\right)\psi(|Q_\pi^Tu|).
\end{equation}
For each permutation $\pi$, reflection of
one coordinate of $Q_\pi^Tu$ preserves its absolute coordinates but
reverses the displayed coefficient. Averaging over the whole sphere
therefore makes each summand vanish, provided the orthoscheme restriction
is integrable. The proof also uses a classical
reflection argument for simplices~\cite{Sah79:Hilberts-Third,KR97:Introduction-Geometric}
to pass from proper rotations to all orthogonal transformations.

The bulk of our arguments are algebraic and combinatorial, with 
the analytical part only requiring the integrability of $\psi$ on the positive part of the unit
sphere. Continuity is a sufficient condition for this, and the condition
can also be deduced from monotonicity.

\paragraph{Relation to previous work and historical remarks}
It is worth pointing out the relationship between our proof and 
the averaging in~\cite{WangWu:Spherical-Hadwiger,lotz2026monotoneinvariantvaluationsconvex}.
The proof in~\cite[Theorem~3.8]{lotz2026monotoneinvariantvaluationsconvex}
uses a signed star subdivision and cancels each fixed-facet apex integral
by reflection in the facet span. In contrast, our flag decomposition iterates the star
identity until every piece is an orthoscheme. This replaces the family of
fixed-facet apex functions by the single parameter function $\psi$, with
the original cone entering only through the orthogonal frames $Q_\pi$.
While both proofs use reflection cancellation, the flag refinement gives it a
uniform coordinate form and restricts the integrability hypothesis to
orthoscheme parameters.

Knoerr's spherical proof~\cite{Knoerr:Spherical-Hadwiger} uses affine
smoothness and gnomonic charts to reduce the problem to the classification
of measurable, translation-invariant valuations on Euclidean polytopes.
Related geometric ingredients occur in the Euclidean result on which it
relies: orthoschemes obtained by successive projections enter a signed
flag representation for simple valuations of degree one, and rotation
averaging produces an invariant odd coefficient function, which must
vanish~\cite[\S4 and Proposition~5.1]{Knoerr:Polytopal-Hadwiger}.
Chen's elementary proof of the Euclidean Hadwiger theorem also uses signed
orthoscheme decompositions obtained by successive
projections~\cite[Lemma~3.3]{Chen:Elementary-Hadwiger}; this decomposition
is invoked in Knoerr's Euclidean
argument~\cite[Lemma~5.2]{Knoerr:Polytopal-Hadwiger}.
Our argument establishes spherical vanishing directly through signed
conic orthoscheme decompositions and apex averaging, requiring only
integrability of the orthoscheme restriction in this step.

Signed orthoscheme decompositions already appear in
Schl\"afli~\cite[\S25, pp.~76--78]{schlaefli:1901}, who writes:
``Im allgemeinen aber kann die Zerlegung auch negative Orthoscheme
enthalten'' (``In general, however, the decomposition may also contain
negative orthoschemes''). His construction follows successive perpendicular
feet along complete flags of faces. Hadwiger~\cite{hadwiger-dissect:1956}
later asked whether every Euclidean simplex admits a finite dissection
into orthoschemes, and explained why the proposed incentre construction
fails when perpendicular feet leave the faces. The problem is solved in
simplex dimensions at most five, with the five-dimensional case established
by Tschirpke~\cite{tschirpke:1994} also in spherical and hyperbolic geometry.
To the best of our knowledge, it remains open in simplex dimension $d\ge6$;
see~\cite[Conjecture~23 and the ensuing discussion, pp.~327--328]{BKKS:2009}.
For spherical simplices in $\sphere{n-1}$, this corresponds to ambient cone
dimension $n\ge7$. Allowing signed pieces
avoids this obstruction and permits the unrestricted apex averaging used
here.

\paragraph{Organization.}
Section~\ref{sec:cones} gives the cone reductions and the $\On n$-$\SOn n$ equivalence,
states Theorem~\ref{thm:simplicial-vanishing}, and deduces the main theorem.
Section~\ref{sec:orth} develops the flag and staircase descriptions of
orthoschemes. Section~\ref{sec:signed-flags} proves the complete signed
flag formula and establishes Theorem~\ref{thm:simplicial-vanishing}.

\paragraph{Notation.}
We write $[n]=\{1,\ldots,n\}$, $\R_+=[0,\infty)$, $\R_{>0}=(0,\infty)$ and
$\R_*=\R\setminus\{0\}$. For a set of vectors, $\pos$ denotes
its nonnegative linear span, and $\lin$ denotes its linear span.
The orthogonal projection onto a subspace $L$ is denoted by $\Pi_L$. 
The symbols $\inter$ and $\relint$ denote interior and relative
interior, respectively. Absolute values of vectors are taken coordinatewise.
Throughout, $\sigma$ is the normalised uniform measure on $\sphere{n-1}$.

\paragraph{Acknowledgements.} The author is grateful to Dennis Amelunxen with whom this projected
originally started, and who provided valuable feedback on the manuscript. 

\section{Convex cones and valuations}\label{sec:cones}
A closed convex cone is a nonempty closed convex set $C\subseteq\R^n$
such that $tC=C$ for every $t>0$. We denote the
space of such cones by $\CC(\R^n)$ and its polyhedral subspace by
$\PP(\R^n)$. A cone is polyhedral if it is generated by finitely many
vectors, and simplicial if the generators can be chosen linearly
independent. The empty list generates $\{0\}$.
A cone is \emph{pointed} if $C\cap(-C)=\{0\}$; the
subspace $C\cap(-C)$ is its \emph{lineality space}. For nonzero cones
we use the Hausdorff distance between their intersections with the unit
sphere, measured with spherical distance. The zero cone is isolated.
In dimension one this gives the discrete topology on
$\{\{0\},\R_+,\R_-,\R\}$. Polyhedral cones are dense in this topology;
see~\cite[Chapter~3]{Sch22:Convex-Cones}.

\subsection{Valuations, subdivisions and intrinsic volumes}
A real \emph{valuation} on $\CC(\R^n)$ satisfies
\[
 \mu(C\cup D)+\mu(C\cap D)=\mu(C)+\mu(D)
\]
whenever $C,D$ and their union are convex cones. The same definition is
used on $\PP(\R^n)$. A valuation is \emph{simple} if $\mu(C)=0$ whenever
$\dim C<n$. In statements about valuations without a continuity
hypothesis, continuity is not implicit in this terminology.

The following fact is the spherical Blaschke selection
theorem~\cite[Theorem~3.1.3]{Sch22:Convex-Cones}, which establishes that 
continuous valuations are bounded.

\begin{lemma}\label{lem:cone-compactness}
The space of nonzero closed convex cones is compact in the spherical
topology. Consequently every continuous real valuation on
$\CC(\R^n)$ is bounded.
\end{lemma}

For simple valuations, we have the following additivity property.
It follows by repeated application of the valuation property, and 
standard facts about simplicial subdivisions.

\begin{lemma}\label{lem:subdivision-additivity}
Let $\mu$ be a simple valuation on $\PP(\R^n)$. If a polyhedral cone
$C$ is subdivided into finitely many polyhedral cones $C_1,\ldots,C_m$
with pairwise disjoint interiors, then
\begin{equation}\label{eq:inclusion-exclusion}
 \mu(C)=\sum_{j=1}^m\mu(C_j).
\end{equation}
Every polyhedral cone admits a simplicial subdivision.
\end{lemma}



The conic intrinsic volumes $v_0,\ldots,v_n$ are special continuous
$\On n$-invariant valuations, and satisfy
\begin{align}
 v_i(L_j)&=\delta_{ij}\quad \text{if } \dim L_j=j,\label{eq:iv-subspaces}\\
 v_n(C)&=\sigma(C\cap\sphere{n-1}),\qquad
 v_i(C)\ge0,\qquad \sum_{i=0}^n v_i(C)=1.\label{eq:iv-volume}
\end{align}
They are intrinsic: an isometric embedding into a larger ambient space
does not change their values, with the extra indices assigned value
zero. These facts follow, for example, from the conic Steiner formula;
see~\cite{AmelunxenLotz:Intrinsic-Volumes-Cones,Sch22:Convex-Cones}.

In dimension one an $\On1$-invariant valuation is determined by its
values on $\{0\}$ and $\R$, since
\[
 2\mu(\R_+)=\mu(\{0\})+\mu(\R).
\]
This proves the one-dimensional $\On1$ classification. The $\SOn1$
version fails because the two rays need not have equal values.

\subsection{Reflection pairing}\label{subsec:reflection-pairing}
The following conic form of the classical incentre dissection is the
reflection-pairing ingredient from scissors congruence; see
Sah~\cite{Sah79:Hilberts-Third} and
Klain--Rota~\cite[Proposition~8.3.1 and the remark following Theorem~11.2.1]{KR97:Introduction-Geometric}. The proof, taken from~\cite[Lemma~3.3]{lotz2026monotoneinvariantvaluationsconvex},
is included for completeness.

\begin{lemma}\label{lem:sah-dissection}
Let $n\ge2$. Every full-dimensional pointed simplicial cone $C$ has a subdivision
into $n(n-1)$ simplicial cones $B_{ij}$, $i\ne j$, such that a hyperplane
reflection interchanges $B_{ij}$ and $B_{ji}$.
\end{lemma}

\begin{proof}
Write $C=\pos(v_1,\ldots,v_n)$ for a basis, let $y_1,\ldots,y_n$ be
its dual basis ($\ip{y_i}{v_j}=\delta_{ij}$), and define
\[
 \nu_i=\frac{y_i}{\|y_i\|},\qquad z=\sum_k\|y_k\|v_k,
 \qquad z_i=z-\nu_i.
\]
Then $z\in\inter C$, $\langle\nu_i,z\rangle=1$, and
\[
 \langle y_k,z_i\rangle
 =\|y_k\|\bigl(1-\langle\nu_k,\nu_i\rangle\bigr).
\]
This is zero for $k=i$ and positive otherwise, so $z_i$ lies in the
relative interior of the facet $F_i=\pos(v_k:k\ne i)$.
Star subdivision at $z$, followed by star subdivision of each facet
at $z_i$, gives the decomposition into the union of cones
\[
 B_{ij}=\pos(z,z_i,v_k:k\notin\{i,j\}),\qquad i\ne j.
\]
For $n=2$ the second subdivision is the identity on each ray.
Reflection $s_{ij}$ in $(\nu_i-\nu_j)^\perp$ interchanges $\nu_i$ and
$\nu_j$, fixes $z$ and the remaining generators, and therefore sends
$z_i$ to $z_j$. Thus $s_{ij}B_{ij}=B_{ji}$.
\end{proof}


\begin{proposition}\label{prop:polyhedral-reflections}
For $n\ge2$, every simple $\SOn n$-invariant valuation $\mu$ on
$\PP(\R^n)$ is $\On n$-invariant.
\end{proposition}

\begin{proof}
For a full-dimensional simplicial cone $C$ and an orientation-reversing
$g$, Lemmas~\ref{lem:subdivision-additivity} and~\ref{lem:sah-dissection}
give
\[
 \mu(gC)=\sum_{i\ne j}\mu(gB_{ij})
 =\sum_{i\ne j}\mu((gs_{ij})B_{ji})
 =\sum_{i\ne j}\mu(B_{ji})=\mu(C),
\]
since $gs_{ij}\in\SOn n$. Subdivision proves the assertion for every
polyhedral cone.
\end{proof}

\begin{corollary}\label{thm:SO=O}
For $n\ge2$, every $\SOn n$-invariant valuation $\mu$on $\PP(\R^n)$ is
$\On n$-invariant. The same holds for continuous valuations on
$\CC(\R^n)$.
\end{corollary}
\begin{proof}
Fix a hyperplane reflection $\rho$ and let
$\eta(C)=\mu(C)-\mu(\rho C)$. This is $\SOn n$-invariant, since that
group is normal in $\On n$. It is simple: if $C\subseteq H$ for a
hyperplane $H$, reflection $\rho_H$ fixes $C$ pointwise, so
$\mu(\rho C)=\mu((\rho\rho_H)C)=\mu(C)$.
Proposition~\ref{prop:polyhedral-reflections} makes $\eta$ reflection
invariant on polyhedral cones. But $\eta(\rho C)=-\eta(C)$, so $\eta=0$
there. This proves the polyhedral case. Continuity and polyhedral
approximation give the statement on all cones.
\end{proof}

\subsection{The residual valuation}
Our results make use of the following condition:
\begin{equation}\tag{A}\label{eq:residual}
 \mu(C)=0\quad\text{if }\dim C<n\text{ or }C\text{ contains a line}.
\end{equation}

\begin{theorem}\label{thm:simplicial-vanishing}
Let $n\ge2$. Every continuous $\On n$-invariant real valuation $\mu$ on
$\CC(\R^n)$ satisfying \eqref{eq:residual} vanishes on every pointed
simplicial cone.
\end{theorem}

\begin{proof}[Proof of Theorem~\ref{thm:main} assuming Theorem~\ref{thm:simplicial-vanishing}]
By Corollary~\ref{thm:SO=O}, it suffices to prove the $\On n$
classification. We use induction on $n$, the one-dimensional case having been
established above. Suppose the classification holds in all smaller
dimensions. For a continuous $\On n$-invariant valuation $\nu$, set
\[
 \mu=\nu-\sum_{i=0}^n\nu(L_i)v_i.
\]
The valuation $\mu$ is continuous, invariant, and zero on every subspace
by \eqref{eq:iv-subspaces}. Its restriction to a proper subspace $V$
is $\On{\dim V}$-invariant and by the induction hypothesis, that
restriction vanishes. The zero-dimensional restriction is zero by normalization. Thus
$\mu$ is simple.

For a fixed line $\ell\subseteq\R^n$, define
$\eta_\ell(D)=\mu(\ell\oplus D)$ on $\CC(\ell^\perp)$.
Orthogonal direct sum with $\ell$ preserves unions, intersections and
continuity, so $\eta_\ell$ is a continuous $\On{n-1}$-invariant
valuation. It vanishes on every subspace, since $\ell\oplus L$ is a
subspace whenever $L\subseteq\ell^\perp$ is. Induction gives
$\eta_\ell=0$. Every cone $C$ containing a line $\ell$ satisfies
\[
 C=\ell\oplus(C\cap\ell^\perp),
\]
so $\mu$ vanishes on such cones and hence satisfies \eqref{eq:residual}.

Theorem~\ref{thm:simplicial-vanishing} now gives zero on every pointed
simplicial cone. By Lemma~\ref{lem:subdivision-additivity}, $\mu$ vanishes
on all polyhedral cones. Their density and continuity give $\mu=0$ on
$\CC(\R^n)$, proving \eqref{eq:main-classification}. Evaluation on
subspaces, using \eqref{eq:iv-subspaces}, proves uniqueness.
\end{proof}

\section{Orthoschemes}\label{sec:orth}
An orthoscheme, or path simplex, is determined by a sequence of mutually
orthogonal steps. It generalises a right triangle to higher-dimensional
simplices~\cite[\S25]{schlaefli:1901}. We use the cones generated by the
successive endpoints of those steps, called orthoscheme cones. 


\subsection{Flags and staircase representations}\label{subsec:orthosch-basic}
A \emph{complete flag} is a sequence
$\mathcal L=(L_0,L_1,\ldots,L_n)$ with
\[
 \{0\}=L_0\subset L_1\subset\cdots\subset L_n=\R^n,
 \qquad\dim L_i=i.
\]
For an apex $u\ne0$, the corresponding \emph{orthoscheme cone} is
\begin{equation}\label{eq:flag-orthoscheme}
 C_{\mathcal L}(u)=\pos(\Pi_{L_1}u,\ldots,\Pi_{L_n}u).
\end{equation}
The pair $(\mathcal L,u)$ is \emph{generic} if
$\Pi_{L_i}u\notin L_{i-1}$ for every $i$. For the standard coordinate
flag we write
\begin{equation}\label{eq:C(a)}
 x_i(a)=\sum_{j=1}^i a_je_j,\qquad
 C(a)=\pos(x_1(a),\ldots,x_n(a)).
\end{equation}
Here $C(a)$ is full-dimensional exactly when $a\in\R_*^n$. Note that if the flag is given by an orthonormal basis $(q_1,\dots,q_n)$,
\begin{equation*}
  L_i = \lin(q_1,\dots,q_i), \quad i\in \{1,\dots,n\},
\end{equation*}
and if $Q=[q_1 \cdots q_n]\in \On n$, then $C_{\mathcal L}(u)=QC(Q^Tu)$. 
The following lemma gives a convenient characterisation of orthoschemes.

\begin{lemma}[Staircase characterisation]\label{lem:staircase-representation}
For a basis $x_1,\ldots,x_n$ of $\R^n$, put
$L_i=\lin(x_1,\ldots,x_i)$ and $x_0=0$. The following are equivalent:
\begin{enumerate}
 \item $\Pi_{L_i}x_{i+1}=x_i$ for $i\in \{1,\dots,n-1\}$;
 \item $x_i=\Pi_{L_i}x_n$ for $i\in \{1,\dots,n\}$;
 \item the steps $x_i-x_{i-1}$ are mutually orthogonal;
 \item the Gram matrix has staircase form
 \begin{equation}\label{eq:staircase-gram}
  \langle x_i,x_j\rangle=\|x_{\min(i,j)}\|^2,
  \qquad i,j\in\{1,\dots,n\}.
 \end{equation}
\end{enumerate}
In this case, the cone $C=\pos(x_1,\ldots,x_n)=QC(a)$ is a generic orthoscheme, where
\begin{equation}\label{eq:staircase-parameters}
 a_i=\|x_i-x_{i-1}\|>0,\qquad
 Q=\left[\frac{x_1}{a_1}\ \frac{x_2-x_1}{a_2}\ \cdots\
                 \frac{x_n-x_{n-1}}{a_n}\right]\in\On n.
\end{equation}
\end{lemma}

\begin{proof}
Condition (1) says that $x_{i+1}-x_i$ is orthogonal to $L_i$.
The preceding steps span $L_i$, so this is equivalent to (3).
Summing the mutually orthogonal steps gives (2) and (4).
Conversely, (2) gives (1) by nesting projections, and (4) gives (1)
by taking inner products with $x_1,\ldots,x_i$.
The vectors defining $Q$ are orthonormal, and their partial sums with
coefficients $a_i$ are exactly the $x_i$. This proves the last assertion.
\end{proof}


\subsection{The parameter function}
Note that $C(Da)=DC(a)$ for any diagonal matrix $D$. In particular,
\begin{equation}\label{eq:orthoscheme-evenness}
 \mu(QC(a))=\mu(C(|a|))
\end{equation}
for an $\On n$-invariant valuation, $Q\in\On n$ and $a\in\R_*^n$. For full-dimensional orthoschemes and orthogonally invariant valuations,
it is thus sufficient to consider positive parameters. 
For an $\On n$-invariant valuation $\mu$, define
\begin{equation}\label{eq:psi}
 \psi(a)=\mu(C(a)),\qquad a\in \R_{>0}^n.
\end{equation}
Positive common scaling of $a$ does not change $C(a)$, and we therefore use the parameter space
\begin{equation}\label{eq:parameter-chamber}
 \Sigma=\sphere{n-1}\cap\R_{>0}^n.
\end{equation}
The map $a\mapsto C(a)$ is continuous on $\Sigma$, so a continuous valuation has a 
continuous parameter function there.

The following positivity property of orthoschemes ensures that an interior
apex gives a positive flag subdivision.

\begin{lemma}\label{lem:positive}
Let $a\in\R_{>0}^n$ and $u\in\inter C(a)$. For every nonzero face $F$ of $C(a)$,
its orthogonal projection $\Pi_{\lin F}(u)$ lies in $\relint F$.
\end{lemma}

\begin{proof}
A standard orthoscheme admits the following representation:
\[
 C(a)=\left\{x\in\R^n\colon \frac{x_1}{a_1}\ge\cdots\ge\frac{x_n}{a_n}\ge0\right\}.
\]
Write $t_i=u_i/a_i$, so $t_1>\cdots>t_n>0$.
A face $F$ is obtained by imposing equalities between neighbouring ratios,
possibly with a terminal block fixed at zero. On each consecutive block
$B$ not fixed at zero, orthogonal projection onto $\lin F$ has coordinates
$(\Pi_{\lin F}u)_i=a_i\bar t_B$, where
\begin{equation}\label{eq:projection-positive}
 \bar t_B=\frac{\sum_{i\in B}a_i u_i}{\sum_{i\in B}a_i^2}
         =\frac{\sum_{i\in B}a_i^2t_i}{\sum_{i\in B}a_i^2}.
\end{equation}
The prescribed terminal block is set to zero. Since these are positive
weighted means of the $t_i$, they remain strictly decreasing from one
block to the next and positive. The projection therefore lies in $\relint F$.
\end{proof}

\section{Signed flag decompositions}\label{sec:signed-flags}
Throughout this section, $\mu$ is a valuation on $\PP(\R^n)$ satisfying
\eqref{eq:residual}. Neither continuity nor invariance is needed until
we express the terminal cones through a single parameter function.
We begin with a signed star subdivision identity for valuations that
satisfy~\eqref{eq:residual}.

\begin{lemma}\label{lem:star}
  Let $C=\pos(v_1,\ldots,v_n)$ for a basis $(v_1,\ldots,v_n)$, and let
  $w=\sum_{i=1}^n\lambda_i v_i\ne0$. 
  If $C_i(w)$ denotes $C$ with generator $v_i$ replaced by $w$, then
  \begin{equation}\label{eq:star}
   \mu(C)=\sum_{i=1}^n\sgn(\lambda_i)\,\mu(C_i(w)),
  \end{equation}
  where we set $\sgn(0)=0$.
\end{lemma}

The proof requires the following elementary observation.

\begin{lemma}\label{lem:flip}
If $v_1,\ldots,v_n$ is a basis, then
\begin{equation}\label{eq:flip}
 \mu\bigl(\pos(v_1,\ldots,-v_i,\ldots,v_n)\bigr)
 =-\mu\bigl(\pos(v_1,\ldots,v_i,\ldots,v_n)\bigr).
\end{equation}
\end{lemma}

\begin{proof}
The union of $\pos(v_1,\ldots,v_i,\ldots,v_n)$ and $\pos(v_1,\ldots,-v_i,\ldots,v_n)$
is a convex cone containing a line, and their intersection is lower-dimensional.
The union and intersection both have value zero, so the valuation identity gives the claim.
\end{proof}

\begin{proof}[Proof of Lemma~\ref{lem:star}]
Set $J=\{i:\lambda_i\ne0\}$, which is nonempty. First suppose
$\lambda_i>0$ for all $i\in J$.
Stellar subdivision of the face
$\pos(v_i\colon i\in J)$ at $w$, retaining the other generators, is the standard star subdivision
of $C$ into the cones $C_i(w)$, $i\in J$, with disjoint interiors; see~\cite[Lemma~11.1.3]{CLS11:Toric-Varieties}.
For general signs, put $s_i=\sgn(\lambda_i)$ for $i\in J$, 
and replace each generator $v_i$ with $s_i v_i$. 
Then $w=\sum_{i\in J}|\lambda_i|(s_i v_i)$
on the cone with reversed generators. By Lemma~\ref{lem:flip}, 
the value of the valuation on the reversed cone is
$(\prod_{j\in J}s_j)\mu(C)$, whereas its $i$th
child has value $(\prod_{j\in J\setminus\{i\}}s_j)\mu(C_i(w))$.
Multiplication by $\prod_{j\in J}s_j$ proves \eqref{eq:star}.
\end{proof}


We now derive an expression for the value of a valuation on a simplicial cone in terms
of associated orthoschemes. 
Fix any basis $v_1,\ldots,v_n$ and $C=\pos(v_1,\ldots,v_n)$.
Consider the corresponding complete flag $\mathcal{L}$ of subspaces
\[
 L_0=\{0\},\qquad L_k=\lin(v_1,\ldots,v_k).
\]
Gram--Schmidt in this order on the matrix with columnes $v_1,\dots,v_n$ gives the orthogonal factor of the QR
decomposition with positive diagonal:
\begin{equation}\label{eq:frame}
 w_k=v_k-\Pi_{L_{k-1}}(v_k),\qquad
 q_k=\frac{w_k}{\norm{w_k}},\qquad
 Q=[q_1\ \cdots\ q_n].
\end{equation}
Note that the projection of a point u onto the face spanned by ${v_1,\dots,v_k}$ is given by 
\begin{equation*}
\Pi_{L_k}(u) = \sum_{i=1}^k \langle u,q_i\rangle q_i = \sum_{i=1}^k \frac{\langle u,q_i \rangle}{\|w_i\|} (v_i-\Pi_{L_{i-1}}(v_i)).
\end{equation*}
In particular, $\langle \Pi_{L_k}(u), q_k \rangle = \langle u,q_k\rangle \langle v_k,q_k\rangle/\|w_k\|$.
Fix an apex $u$, and define
\begin{equation}\label{eq:projected-flag}
 z_k=\Pi_{L_k}(u)=\sum_{j=1}^k \langle q_j,u\rangle q_j, \qquad K(u)=\pos(z_1,\ldots,z_n).
\end{equation}
If every $\langle q_j,u\rangle \ne0$, then $K(u)=QC(Q^Tu)$ is a full-dimensional orthoscheme,
with positive parameters $|Q^Tu|$ up to congruence.

For a permutation $\pi\in S_n$, a subscript $\pi$ denotes the same construction applied
to the ordered generators $(v_{\pi(1)},\ldots,v_{\pi(n)})$, with $u$ fixed.
Also write $S_{\pi,k}=\{\pi(1),\ldots,\pi(k)\}$.
We call $u$ generic if $\ip{u}{q_{\pi,k}}\ne0$ for every
$\pi,k$. Note that the excluded set of non-generic $u$ is a finite union of hyperplanes.

\begin{theorem}\label{thm:flag}
Let $\mu$ be an $\On n$-invariant valuation that satisfies~\eqref{eq:residual}. Then for every generic $u\in\R^n$,
\begin{equation}\label{eq:flag}
 \mu(C)=\sum_{\pi\in S_n}
  \left(\prod_{k=1}^n\sgn\bigl((Q_\pi^Tu)_k\bigr)\right)
  \psi(|Q_\pi^Tu|).
\end{equation}
\end{theorem}

\begin{proof}
For each nonempty $S\subseteq\{1,\ldots,n\}$ set $L_S=\lin(v_i\colon i\in S)$ and write
\[
 p_S=\Pi_{L_S}(u)=\sum_{i\in S}\alpha_i^S v_i.
\]
If $S=S_{\pi,k}$ and $i=\pi(k)$, taking the scalar product with
$q_{\pi,k}$ gives
\begin{equation}\label{eq:coefficient}
 \alpha_{\pi(k)}^{S_{\pi,k}}
   =\frac{\ip{u}{q_{\pi,k}}}{\norm{w_{\pi,k}}}\neq 0,
\end{equation}
where the fact that these entries are nonzero follows from genericity.

Apply Lemma~\ref{lem:star} to $C$, with apex
$p_{\{1,\ldots,n\}}=u$:
\begin{equation*}
  \mu(C)=\sum_{i=1}^n\sgn(\alpha_i^{[n]})\mu(C_i(u)).
\end{equation*}

In the term that replaced $v_i$, apply Lemma~\ref{lem:star} with apex
$p_{[n]\setminus\{i\}}$. This replaces each remaining original generator
$v_j$ in turn, with coefficient $\sgn\alpha_j^{[n]\setminus\{i\}}$;
the coefficient of the retained generator $u$ is zero.
Continue this procedure downwards. Along the branch
indexed by $\pi$, the cone at stage $k$ has generators
\begin{equation}\label{eq:intermediate}
 \{v_i\colon i\in S_{\pi,k}\} \cup\
 \{p_{S_{\pi,k+1}},\ldots,p_{S_{\pi,n}}\}.
\end{equation}
The first set of generators spans $L_{\pi,k}$, and the
successive projected vectors have nonzero components $\langle u,q_{\pi,j}\rangle$
in the complementary flag directions $q_{\pi,j}$, $j>k$, hence the set of generators~\eqref{eq:intermediate}
forms a basis. At $k=1$, the remaining original ray is replaced by
$p_{S_{\pi,1}}$ and the one-generator identity contributes
$\sgn\alpha_{\pi(1)}^{S_{\pi,1}}$.

There are exactly $n!$ branches, one for each order of deletion, and
the resulting terminal cone is $K_\pi(u)$. Its coefficient is
\[
 \prod_{k=1}^n\sgn\alpha_{\pi(k)}^{S_{\pi,k}}
 =\prod_{k=1}^n\sgn\ip{u}{q_{\pi,k}},
\]
by \eqref{eq:coefficient}.
Orthogonal invariance gives \eqref{eq:flag}.
\end{proof}


Let $\mu$ be an $\On n$-invariant valuation on $\PP(\R^n)$ satisfying
\eqref{eq:residual}, and let $\psi$ be its orthoscheme restriction. In what 
follows, assume that $\psi|_{\Sigma}\in L^1(\Sigma,\sigma)$, that is, $\psi|_{\Sigma}$ is measurable and
\begin{equation}\label{eq:integrability}
 \int_\Sigma|\psi(a)|\,\mathrm{d}\sigma(a)<\infty.
\end{equation}
For $\eps\in\{-1,1\}^n$ write $D_\eps=\diag(\eps_1,\ldots,\eps_n)$
and $s(\eps)=\prod_k\eps_k$.

\begin{theorem}\label{thm:average}
For every full-dimensional simplicial cone $C$ and $\On n$-invariant valuation $\mu$ on $\PP(\R^n)$ satisfying \eqref{eq:residual}, such that $\psi|_{\Sigma}\in L^1(\Sigma,\sigma)$ for $\psi(a)=\mu(C(a))$,
\begin{equation}\label{eq:average}
 \mu(C)=0. 
\end{equation}
Moreover, $\mu$ vanishes on all polyhedral cones.
\end{theorem}

\begin{proof}
Define, for $a\in\sphere{n-1}$ with all coordinates nonzero,
\[
 H(a)=\left(\prod_{k=1}^n\sgn a_k\right)\psi(|a|).
\]
Set $H=0$ on the coordinate hyperplanes. Then $H\in L^1(\sphere{n-1},\sigma)$ and \eqref{eq:flag} gives
$\mu(C)=\sum_\pi H(Q_\pi^Tu)$ for almost every $u$. Integrating over the sphere, we get
\begin{align*}
  \mu(C) &= \sum_{\pi\in S_n}\int_{\sphere{n-1}}H(Q_\pi^Tu)\,\mathrm{d}\sigma(u) \\
  &= n!\int_{\sphere{n-1}}H(u)\,\mathrm{d}\sigma(u) \\
  &= n!\sum_{\eps\in\{-1,1\}^n}\int_\Sigma H(D_\eps a)\,\mathrm{d}\sigma(a) \\
  &= n!\left(\sum_{\eps\in\{-1,1\}^n}s(\eps)\right)
       \int_\Sigma\psi(a)\,\mathrm{d}\sigma(a)=0.
\end{align*}
Here we used orthogonal invariance of $\sigma$, split the sphere into its
orthants, and used $H(D_\eps a)=s(\eps)\psi(a)$ on $\Sigma$.
The sign sum is zero. The claim for all polyhedral cones
follows by simplicial subdivision.
\end{proof}

\begin{proof}[Proof of Theorem~\ref{thm:simplicial-vanishing}]
By Lemma~\ref{lem:cone-compactness}, $\mu$ is bounded, and its restriction
to the positive orthoscheme chamber is continuous. Thus
\eqref{eq:integrability} holds. Theorem~\ref{thm:average} gives zero on
every full-dimensional pointed simplicial cone.
\end{proof}

\bibliographystyle{alpha}
\bibliography{refs}
\end{document}